\documentclass{article}
\usepackage{amssymb}
\usepackage{amsthm}
\usepackage{amsmath}
\usepackage{fullpage}
\usepackage{mathtools}
\usepackage{graphicx}
\usepackage{amsmath,amssymb,amsthm,bbm}
\usepackage{bbm}
\usepackage{indentfirst}
\usepackage{authblk}
\usepackage{hyperref}
\DeclareMathOperator{\Tr}{Tr}
\newtheorem{theorem}{Theorem}
\newcommand{\inprod}[2]{\left\langle #1, #2 \right\rangle}
\newtheorem*{remark}{Remark}
\usepackage{color}
\providecommand{\R}{\mathbb R}
\providecommand{\E}{\mathsf E}

\numberwithin{equation}{section}
\newtheorem{thm}{Theorem}[section]

\newtheorem{lem}[thm]{Lemma}

\newtheorem{prop}[thm]{Proposition}

\newtheorem{coro}[thm]{Corolary}
\theoremstyle{definition}
\newtheorem*{rem}{Remark}

\def\norm#1{\Vert#1\Vert}

\def\norm#1{\Vert#1\Vert}
\def\abs#1{|#1|}

\usepackage{cite}
\title{Constructing exchangeable pairs by diffusion on manifolds and its application}
\author{Weitao Du\thanks{{ weitao.du@northwestern.edu}. }}
\date{}
\begin{document}
\maketitle

\begin{abstract}

We construct a continuous family of exchangeable pairs by perturbing the random variable through diffusion processes on manifolds in order to apply Stein method to certain geometric settings. We compare our perturbation by diffusion method with other approaches of building exchangeable pairs and show that our perturbation scheme cooperates with the infinitesimal version of Stein's method harmoniously. More precisely, our exchangeable pairs satisfy a key condition in the infinitesimal Stein's method in general. Based on the exchangeable pairs, we are able to extend the approximate normality of eigenfunctions of Laplacian on a compact manifold to eigenfunctions of Witten Laplacian, which is of the form:$\Delta_w = \Delta - \nabla H$. We then apply our abstract theorem to recover a central limit result of linear statistics on sphere. Finally, we prove an an infinitesimal version of Stein's method for exponential distribution and combine it with our continuous family of exchangeable pairs to extend an approximate exponentiality result of $|\Tr U|^2$, where $\Tr U$ is the trace of the first power of a matrix $U$ sampled from the Haar measure of unitary group, to arbitrary power and its analog for general circular ensemble.
\end{abstract}

\section{Introduction} \label{intro}
In \cite{cs72}, Charles Stein introduced the celebrated Stein method along with the notion of exchangeable pairs for proving the rate of convergence for normal distribution. The main idea behind Stein's method is to replace the characteristic function typically used to show distributional convergence with a characterizing operator $\mathcal{A}$. This characterizing operator of the target distribution defines a differential equation, the Stein equation, which can be solved and has a probabilistic representation. Then the problem of  bounding the distance between two distributions is reduced to estimate the derivatives of Stein equation's solution. The metric we use to measure the distance is usually of the form:
$$d(\mu , \nu ) = \sup_{h\in \mathcal{H}} |  \mathbb{E} h(X) -  \mathbb{E}h(W)|\ ,$$
where $X$, $W$ are random variables with distribution $\mu$ and $\nu$ and $\mathcal{H}$ is a certain function class. We name a few distributions that have been tackled by Stein method: the Poisson distribution \cite{chen75,cdm05,hg}, the exponential distribution \cite{cdm05,cfr}, the Beta distribution \cite{beta} and manifold-valued measures \cite{cdm05}. In \cite{lam1}, the author proved a CLT for general linear statistics of the circular $\beta$-ensembles and the proof is based on the "loop equation" method which bears a resemblance with Stein's method. For a short survey, the reader can refer to \cite{cha}.

 In this article, we focus on building exchangeable pairs in order to apply (the infinitesimal) Stein method efficiently. Recall that a pair $(X,X')$ of random elements on a common probability space is exchangeable if
$$(X,X') \overset{d}{=} (X' , X)\ .$$
One standard way to construct an exchangeable pair is to go one step in a discrete time reversible Markov chain. This works well when the underlying problem is essentially discrete such as a random sequence with weak correlation. On the other hand, when there exists a continuous symmetry in the model, we can create a continuous family of exchangeable pairs by perturbing the random variable we want to study continuously In \cite{meckes-thesis}, the author developed an infinitesimal version of Stein's method which is coupled with a continuous family of exchangeable pairs. It was applies to several geometric settings like the sphere harmonics \cite{meckes-L} and the orthogonal group \cite{M07},\cite{wwr}. The procedure to produce the exchangeable pairs in \cite{M07,meckes-L,El2} was to perturb the original random variable $X$ along a deterministic flow such as geodesic flows. In \cite{Fj,Fj10,jn}, the authors constructed the continuous family of exchangeable pairs by the heat kernel. Once we have the general set-up, the next step boils down to control the moments of certain functionals of the target random variable. This task is  usually accomplished by combinatorial techniques.For instance, the so-called Weingarten calculus has been developed to calculate moments of a broad class of functions defined on the lie group. This kind of calculation also has a close relation with integrable probability, see \cite{al,bs}.

The main purpose of the present paper is to show how to create a continuous family of exchangeable pairs by diffusion process on manifold. When we set underlying process as Brownian motion, this method can be seen as a microscopic explanation of the heat kernel method in \cite{Fj10}. For example, the heat kernel perturbation was used in \cite{ja} to study linear eigenfunctions of  Laplacian operator on unitary group.  Besides covering the pure Laplacian case, our method also works for the Witten Laplacian by adding a drift generated by a gradient vector field $\nabla H$ to the Brownian motion, Where $H(x)$ is a weight function on the manifold. We would like to point out the technique in \cite{lam2}, which is quite different from our approach, can also be generalized to the Gibbs measures on manifolds. On the other hand, the diffusion perturbation scheme matches well with the infinitesimal Stein method. We will demonstrate this point by showing that a key condition we need to check to guarantee the approximate normality holds in general under our framework. We list the third condition of the  infinitesimal version of Stein's method (see \cite{El2}) as follows:
$$\lim_{t \rightarrow 0} \frac{1}{t} \E [ | W_t - W |^2  \mathbb{I}( | W_t - W |^2 > \rho ) ] = 0  \ .$$
We prove this by utilizing a concentration property of Brownian motion as $t \rightarrow 0$ and Girsanov transform. Note that this condition is usually verified by calculating the fourth moment, which is a heavy computational task even if there is an explicit formula available. Moreover, this condition is universal in the sense that it also emerges in infinitesimal Stein method for other probabilistic distributions. To show  how this universal condition appears, we convert the Stein method for exponential distribution in \cite{jn} into an infinitesimal version. Then we are able to study the distribution of $|\Tr (U^k)|^2$ for all $k > 0$, where $U$ is sampled by the Haar measure of unitary group. This is closely related with the central limit phenomenon of unitary matrices (see \cite{kg}).  Properties of the norm of $\Tr (U^k)$ plays a role in comparing the circular unitary ensemble and the Riemann zeta function (see \cite{fm,cd}), where a precise bound on the rate of convergence of  $|\Tr (U^k)|$ is crucial. Furthermore, the analog of the trace power functions can be defined for general circular ensemble which is a key ingredient for studying linear statistics of the corresponding ensemble. Although the geometry of the underlying space has transformed from the compact group to an open simplex, we can still adapt our method to this case. Note that for general circular ensemble, the corresponding diffusion has a drift, so we need to harness the full power of the abstract theorem developed in section 3.  Last but not least, our method may have the potential to generalize to metric measure space where the notion of diffusion can still be defined (see \cite{vjr,ks}).

The rest of the contents are organized as follows. We first introduce the notation in section 1.1. In section 2 and 3, the method of constructing exchangeable pairs by diffusion on manifold is given and then we prove a general theorem on the approximate normality of eigenfunctions of the Witten Laplacian, which is also known as the Bakery-Emery Laplacian (see \cite{xiang12,dim}). The calculation in the process of proving the theorem will be encountered again in the following sections. Section 3 recovers a result in \cite{meckes-thesis} by implementing the Brownian motion perturbation. The next section treats the norm of the trace of the powers of unitary matrix. We first prove an  infinitesimal approximation theorem on the basis of the estimates of the Stein equation built in \cite{jn}. Then we apply the theorem to prove the approximate exponentiality of $|\Tr (U^k)|^2$ for general $k > 0$, where $U$ is sampled by the Haar measure of unitary group. In the last section, we modify the general theorem in section 2 in order to handle Dyson Brownian motion, then we apply the theory to study the analog of  $|\Tr (U^k)|^2$ for general circular ensemble.
\subsection{Notation}
Let $(M,g)$ be an n-dimensional Riemannian manifold. We assume $M$ is compact if not particularly indicated. In local coordinates $(x^1, \dots, x^n)$ with tangent vectors $\{\frac{\partial}{\partial x_i} \}_{i=1}^{n}$, define
$$(G(x))_{ij} = g_{ij} (x) = g( \frac{\partial}{\partial x_i} , \frac{\partial}{\partial x_j}),\ \ \ \ g^{ij}(x) = (G^{-1} (x))_{ij}\ .$$
Denote the covariant derivative on $M$ as $\nabla$. Then for a smooth function $H(x)$ (the potential function) on $M$, the gradient vector field is defined as $g^{-1} \nabla H$. We will abuse the notation and denote the gradient vector field as $\nabla H$ as well. In local coordinates, the Christoffel symbols $\Gamma^{m}{}_{ij}$ are given by
\begin{equation}
\Gamma^{m}{}_{ij}={\frac{1}{2}}\,g^{mk}\left({\frac {\partial }{\partial x^{j}}}g_{ki}+{\frac {\partial }{\partial x^{i}}}g_{kj}-{\frac {\partial }{\partial x^{k}}}g_{ij}\right) \nonumber
\end{equation}
and the components of the Ricci curvature tensor are given by
\begin{equation}
Ric_{ij}={\frac {\partial \Gamma ^{\ell }{}_{ij}}{\partial x^{\ell }}}-{\frac {\partial \Gamma ^{\ell }{}_{i\ell }}{\partial x^{j}}}+\Gamma ^{m}{}_{ij}\Gamma ^{\ell }{}_{\ell m}-\Gamma ^{m}{}_{i\ell }\Gamma ^{\ell }{}_{jm}. \nonumber
\end{equation}
For a general tensor $T$ on $M$ or a certain domain specified by the context, we denote the norm of $T$ induced by the metric $g$ as $\norm{T}$ and the maximum of the norm on the domain as $$\norm{T}_{max}\ .$$\\
Let $ \Delta$ be the Laplacian operator with respect to g, in local coordinates, we have
$$\Delta f = \Tr \nabla \nabla f = \frac{1}{\sqrt{det(G)}} \frac{\partial}{\partial x_j} (\sqrt{det(G)} \cdot g^{jk} \frac{\partial f}{\partial x_k}) \ .$$
The Witten Laplacian with respect to $H(x)$ is defined by
$$\Delta_w : = \Delta - \nabla H\ .$$
We denote the orthonormal frame bundle of $M$ as $\mathcal O (M)$, which is the set of $(x, \phi^1,\dotsc,\phi^n)$, where $x \in M$ and $\phi^1,\dotsc,\phi^n$ forms an orthonormal basis for $T_x M$. Then the covariant derivative $ \nabla$ induces a connection 1-form $\omega$ on $\mathcal O (M)$. It's well known that the tangent bundle of $\mathcal O (M)$ splits by the connection:
$$T \mathcal O (M) = H T\mathcal O (M) + V T \mathcal O (M) \ ,$$
where we denote the vertical subbundle as $V T \mathcal O (M)$ and the horizontal subbundle, which is the kernel of $\omega$ as $H T\mathcal O (M)$.

On the space of real $n\times n$ matrices, 
the Hilbert-Schmidt inner
product and the corresponding norms  are defined by
$$\inprod{A}{B}_{H.S.}=\Tr(AB^T),\ \ \ \ \norm{A}_{H.S}=\sqrt{\Tr(AA^T)}\ .$$
The operator norm of a matrix $A$ over $\R$ is defined by
$$\norm{A}_{op}=\sup_{|v|=1,|w|=1}|\inprod{Av}{w}|.$$

Let X,Y be random variables, the Wasserstein distance $d_W (X,Y)$ is defined by
$$d_W (X,Y) = \sup_{Lip(g) \leq 1} | \mathbb{E} g(X) - \mathbb{E} g(Y)|\ ,$$
where $Lip(g) = \sup_{x \neq y} \frac{|\mathbb{E} g(X) - \mathbb{E} g(Y)|}{|x-y|}$ is the Lipschitz constant of $g$. 

The total variation distance between X and Y is defined by
$$d_{TV} (X,Y) = \frac{1}{2} \sup_{g} |\mathbb{E} g(X) - \mathbb{E} g(Y)|\ ,$$
where $g$  belongs to the class which consists of continuous functions which are bounded by 1.

The Kolmogorov distance $d_K (X,Y) $ is defined by
$$d_K (X,Y) = \sup_{s \in \mathbb{R}}| \mathbb{P}( X \leq s)- \mathbb{P}( Y \leq s)|\ .$$

\section{Building the exchangeable pair}
Suppose $X$ is a random point in $M$. To create an exchangeable pair, we run a stochastic process starting at $X$. The specific stochastic process  we will consider is defined by a stochastic differential equation on the orthonormal frame bundle $\mathcal O (M)$. We will briefly review the construction of stochastic differential equation on manifold, for the detail, the reader can see \cite{Hsu, fyw}. 

Let $\pi_*$ be the be the isomorphism $\pi_* :  H T\mathcal O (M) \rightarrow M$ induced by the canonical projection $\pi: \mathcal O (M) \rightarrow M$. Suppose we are given an element $\phi = (x, \phi^1,\dotsc,\phi^n)$ of $\mathcal O (M)$ and a vector $v \in T_x(M)$,  then the horizontal lift of $v$ at $\phi$ is the unique tangent vector $\widetilde{v} \in  H T\mathcal O (M)$ s.t 
$$\pi_* \widetilde{v} = v\ .$$
 Let $
\mathcal H^i ( \phi) = \widetilde{\phi^i} $ be the horizontal lift of the orthonormal basis of $T_x M$, then we can define the horizontal Laplacian as
$$\Delta_{\mathcal H} = \sum_{i=1}^n (\mathcal H^i)^2 \ .$$
It has the following relation with the Laplacian on $M$:
$$\Delta_{ \mathcal H } ( f \circ \pi ) = \Delta (f) \circ \pi \ .$$
The stochastic differential equation is of the form:
\begin{equation}
\left\{
\begin{array}{l}
d U_t = \sum_{i=1}^d \mathcal H^i ( U_t ) \circ d B^i_t  - \frac{1}{2}  \widetilde{\nabla H} (U_t) dt  ,   \\
\pi (U_0) = X. 
\end{array}
\right.
\end{equation}
The vector field $\widetilde{\nabla H}$ on orthonormal frame bundle is the horizontal lift of $\nabla H$ defined on $M$. We denote the projection of $U_t$ to $M$ as $X_t$, i.e $X_t = \pi (U_t)$. Since $M$ is compact, everything is smooth and bounded which is enough to insure the existence of a unique strong solution to (2.1).

The stationary distribution of this process is $d\mu = \frac{1}{Z} e^{-H(x)} dx$, where Z is the normalization factor to make $d \mu $ a probabilistic measure. Since $X_t$ is reversible with stationary distribution $\mu$, we know that $(X_t , X)$ has the same law as $(X , X_t )$ If we sample $X$ from the stationary distribution $\mu$. In this way, we have created a family of exchangeable pairs $(X_t , X)$ indexed by the time t. 

There is a correspondence between the solution of the heat equation of Witten Laplacian and the underlying process $X_t$. Consider the following initial value problem for $\theta = \theta ( t,x)$:
\[
\left\{
\begin{array}{l}
\frac{\partial \theta}{\partial t} = \frac{1}{2} \Delta_w \theta ,  \\
\theta (0,x)  = f(x) , \ \ \ x \in M. 
\end{array}
\right.
\]
Then we have a Feynman-Kac representation $$ \theta (t,x) = \mathbb{E} [ f ( X_t) | X_0 = x ] \ ,$$ for the detail, see Theorem 7.2.1 of \cite{Hsu}. On the other hand, from functional analysis, we know that the solution has a semigroup representation:
$$ \theta (t, x) = e^{t \Delta_w } f(x)\ , $$
where $e^{t \Delta_w}$ is defined as $e^{t \Delta_w} : = I + t \Delta_w + t^2 \frac{\Delta_w^2}{2!} + \cdots$. 

We will need the asymptotics of $\theta (t, x) = \mathbb{E} [ f ( X_t) | X_0 = x ]$ with respect to the time $t$. In \cite{ja} and \cite{Fj10}, this was down by means of the semigroup representation of heat equation. We give a probabilistic argument that also applies when the drift of the underlying process is singular. By Ito's formula,
$$\mathbb{E} [ f ( X_t) | X_0 = x ] = f(X_0) + \mathbb{E} [ \int_0^t \Delta_w f(X(s)) ds | X_0 = x ]\ . \eqno{(2.2)}$$
Since we assume $f$ is a smooth function on a compact manifold $M$, we can exchange the order of differentiation and expectation freely, 
$$\lim_{t \rightarrow 0}  \mathbb{E} [ f ( X_t) - f(X_0) | X_0 = x ] = \Delta_w f(	X_0)\ .  \eqno{(2.3)}$$

\begin{rem}\mbox{}
\begin{enumerate}
\item When we consider noncompact manifolds, we usual require a lower bound on the Bakry-Emery Ricci tensor associated with the Witten Laplacian:
$$ Ric_H : = Ric + \nabla^2 H\ ,$$
where $ \nabla^2 H$ is the Hessian two tensor of the potential function $H$. Otherwise, we may encounter analytical issues ( see\cite{fyw} ).
\item In \cite{meckes-L}, the way to create the exchangeable pair is by perturbing $X$ by a geodesic flow. In that setting, we have
\begin{equation}\nonumber
\begin{split}\label{taylor}
f(X_\epsilon)-f(X)&=\epsilon\cdot\left.\frac{d(f\circ\gamma)}{dt}\right|_{t=0}+
\frac{\epsilon^2}{2}\cdot\left.\frac{d^2(f\circ\gamma)}{dt^2}\right|_{t=0}
+O(\epsilon^3)\ .
\end{split}\end{equation}  
The first order term in $\epsilon$ is cancelled after taking expectation due to the symmetry. So the actual perturbing order in \cite{meckes-L} is $\epsilon^2$. However after adding the drift to the Laplacian, the symmetry is broken. So the original perturbation doesn't work in our case. By Ito's formula,
$$df(X_t) = \sum_{i=1}^d ( \mathcal H^i f ) d B^i_t + [\frac{1}{2} \Delta f - \frac{1}{2} (\nabla H f ) ] dt\ , \eqno{(2.4)}$$
 the Laplacian $\Delta$ and the drift $\nabla H$ scales in same order. As we will see in the next section, this is important if we want to apply the infinitesimal Stein method.
\end{enumerate}
\end{rem}

\section{Approximate normality of eigenfunctions of Witten Laplacian}
We need the following infinitesimal multivariate normal approximation theorem introduced in \cite{El2}:
\begin{thm}\label{abscont}Let $W$,$W_t$ be $ \mathbb{R}^d$ -valued random vector and  $(W,W_t)$ forms a family of 
exchangeable pairs defined on a common probability space.  Suppose there is an invertible matrix $\Lambda$, a positive definite symmetric  matrix $\Sigma$, a  random vector $E$, a random matrix $E'$ such that

\begin{enumerate}
\item $$\E\left[W_t -W\big|
W\right]=-t \Lambda W+t E + O(t^2),$$\label{lindiff3}
\item $$\E\left[(W_t-W) (W_t - w)^T\big|
W\right]=2 t \Lambda \Sigma  + t E' + O(t^2),$$\label{quaddiff3}
\item For each $ \rho > 0$,$$\lim_{t \rightarrow 0} \frac{1}{t} \E [ | W_t - W |^2  \mathbb{I}( | W_t - W |^2 > \rho ) ] = 0  $$\label{cubediff3}
\end{enumerate}
Then
$$d_{W}(W,\Sigma^{\frac{1}{2}}Z)\le \norm{\Lambda^{-1}}_{op}\left[\frac{1}{2} \norm{\Sigma^{-\frac{1}{2}}}_{op}
 \mathbb{E}\norm{E'}_{H.S}+ \mathbb{E}\left|E\right|\right],$$
where $Z$ is a standard Gaussian random vector in $\mathbb{R}^d$.
\end{thm}
\begin{remark}
This approximation theorem is the version in \cite{El2}, which is different from \cite{meckes-L} in the third condition. In \cite{meckes-L}, the third condition is
$$ \frac{1}{t} \mathbb{E} |W_t-W|^3 = o(1)\ .$$
It's hard to bound the third moment directly by our perturbation method. On the other hand, due to the concentration of the Gaussian measure, $\mathbb{P} ( | W_t - W |^2 > \rho ) $ goes to zero exponentially in t, the third condition in  \cite{El2} holds as a trivial corollary.

The function $f: M \rightarrow \mathbb{R}$ is an eigenfunction of $\Delta_w$ with eigenvalue $-\nu$ if 
$$ \Delta_w f = \Delta f(x) - \frac{1}{2} \nabla H f(x) = - \nu f(x)\ .$$
We can always normalize $f(x)$ such that $\int_M f(x) d \mu (x) = 0$ and $\int_M f^2(x) d \mu (x) = 1$.
\end{remark}
Let $X$ be a random point of $M$ sampled from measure $d\mu$. Define the value distribution of $f$ with respect to $\mu$ as the distribution of the random variable $f(X)$. 
Consider a sequence of $L^2$-orthonormal (with respect to $d\mu (x)$) eigenfunctions of $\Delta_w $ with corresponding eigenvalues $- \nu_i$. Let $W$ be the random vector $(f_i(X))$. We construct $W_t$ as $(f_i(X_t))$, where $X_t$ is the stochastic process starting at $X$ defined previously. Then $( W , W_t )$ forms a family of exchangeable pairs parameterized by time t.

\begin{lem}
Let $W_t$, $W$ be defined as above. Then
$$\lim_{t \rightarrow 0} \frac{1}{t} \mathbb{E} [ W_t - W | W ] = - \nu_i f_i(X)\ .  \eqno{(3.1)}$$
\end{lem}
\begin{proof} \ \ Applying (2.3), one has that
\nonumber
\begin{align} 
\lim_{t \rightarrow 0} \frac{1}{t} \mathbb{E} [ W_t - W | W ] &= \frac{1}{t} \{ \mathbb{E} [  f_i( X_t ) | X ] - f_i(X) \} \\
& = \lim_{t \rightarrow 0} \frac{1}{t} [ f_i(X) + t \Delta_w f_i(X) + O(t^2) - f_i(X)  ] \\
& = - \nu_i f_i(X)
\end{align}
\end{proof} 

This implies that condition 1 of Theorem 3.1 is satisfied if we take $\Lambda = diag (\nu_1, \cdots , \nu_k ) $ and $E = 0$. Then the operator norm of $\norm{\Lambda^{-1}}$ satisfies
$$\norm{\Lambda^{-1}}_{op} \leq \max_{1\le i\le k} (\mu_i^{-1})\ .$$
\begin{lem}
Let $\{ f_i(x) \}$ be any sequence of smooth functions on $M$ (not necessarily eigenfunctions).  Let $W$,$W_t$  be the random vector $(f_i(X))$ and  $(f_i(X_t))$ as above. Then
$$ \mathbb{E} [ (W_t - W)_i (W_t - W)_j | W ] = 2 t \langle \nabla f_i(X) , \nabla  f_j(X) \rangle + O(t^2)\ ,           \eqno{(3.2)}$$
where $\langle \cdot , \cdot \rangle$ is the inner product induced by the metric g.
\end{lem}
\begin{proof} \nonumber
\begin{align}
 \mathbb{E} [(W_t - W)_i (W_t - W)_j  | W ] &=  \mathbb{E} [ f_i(X_t)f_j(X_t) | W] -  W_i  \mathbb{E} [ f_j(X_t) | W] -  W_j  \mathbb{E} [ f_i(X_t) | W]+ f_i(X_t)f_j(X_t)  \\
& =f_i(X)f_j(X) + t\Delta_w [f_i(X)f_j(X)] - 2f_i(X)f_j(X) - t f_i(X)\Delta_w f_j(X) \\
&\ \ \ - tf_j(X)\Delta_w f_i(X) + f_i(X_t)f_j(X_t) + O(t^2) \\
& = t\Delta_w [f_i(X)f_j(X)]  - tf_i(X)\Delta_w f_j(X) - tf_j(X)\Delta_w f_i(X) + O(t^2) \ .
\end{align}
Plugging in the identity
$$\Delta_w [f_i(x)f_j(x)] = f_i(x)(\nabla H f_j(x) + \Delta f_j(x)) + f_j(x)(\nabla H f_i(x) + \Delta f_i(x)) + 2 \langle \nabla f_i(x) , \nabla  f_j(x) \rangle\ ,$$
we get the final expression
$$
 \mathbb{E} [(W_t - W)_i (W_t - W)_j  | W ]= 2 t \langle \nabla f_i(X) , \nabla  f_j(X) \rangle + O(t^2) \ .$$
\end{proof} 

So condition 2 of Theorem 3.1 is satisfied if we take  $E' = 2 \langle \nabla f_i(X) , \nabla  f_j(X) \rangle - 2\Lambda $ and $\Sigma$ to be the identity matrix. Although it's hard to estimate $\langle \nabla f_i(X) , \nabla  f_j(X) \rangle$ pointwise,  we can calculate the expectation:
\begin{align} \nonumber
\mathbb{E} \langle \nabla f_i(X) , \nabla  f_j(X) \rangle & = \sum_{k=1}^d \int_{M}\langle \nabla_k f_i(x) , \nabla_k f_j(x) \rangle d\mu \\ \nonumber
& = \sum_{k=1}^d \int_{M} \langle \nabla_k f_i(x) , \nabla_k f_j(x) \rangle e^{-H(x)} dx\ . \nonumber
\end{align}
Since $M$ has no boundary, after integrating by parts, we have
\begin{align} \nonumber
\mathbb{E} \langle \nabla f_i(x) , \nabla f_j(x) \rangle & = - \int_{M} \langle f_i(x) , \Delta f_j (x) \rangle d\mu + \int_{M} \langle f_i(x) , \nabla H f_j(x) \rangle d \mu \\ \nonumber
& = -  \int_{M} \langle f_i(x) , \Delta_w f_j(x) \rangle d \mu \\   \nonumber
& = \delta_{ij} \nu_j\ , \nonumber
\end{align}
thus,
$$\mathbb{E} \norm{E'}_{H.S} = 2  \mathbb{E} \sqrt{\sum_{i,j=1}^k [ \langle \nabla f_i(X) , \nabla  f_j(X) \rangle - \mathbb{E} \langle \nabla f_i(X) , \nabla  f_j(X) \rangle ]}\ .$$

To verify condition 3, we need proposition 5.1.4 of \cite{Hsu} on the exit time of Brownian motion on manifold. Let $\tau_r$ denotes the first exiting time of the ball $B(x , r)$ of the Brownian motion starting at x. We have the following property:
\begin{lem}
Let $i_M$ be the injectivity radius of M and $r < i_M$. Then $\exists$ a positive smooth function $C_r (x)$ s.t as $t \rightarrow 0$, 
$$ \mathbb{P}_x \{ \tau_r \leq t \} \sim \frac{C_r (x)}{t^{(d-2)/2}} e^{-r^2 / 2t}\ ,$$
uniform in M.
\end{lem}
Once we have this key lemma on the exit time, we are ready to verify condition (3).
\begin{lem}
For every $ \rho > 0$,$$\lim_{t \rightarrow 0} \frac{1}{t} \E [ | W_t - W |^2  \mathbb{I}( | W_t - W |^2 > \rho ) ] = 0\ .  $$
\end{lem}
\begin{proof} 
For simplicity, we assume $k =1$.  For general random vector $W = \{f_i(X)\}_{i=1}^k$, the proof is almost the same. We first prove the theorem when there is no drift and extend the argument  by Girsanov transform.
\begin{enumerate}
    \item $\nabla H = 0$, $\Delta_w = \Delta\ .$\\
When $\nabla H = 0$, the underlying process is just Brownian motion. We can write the expectation in the following form:
$$ \lim_{t \rightarrow 0} \frac{1}{t} \mathbb{E} [ | W_t - W |^2  \mathbb{I}( | W_t - W |^2 > \rho ) ] = \lim_{t \rightarrow 0}\frac{1}{t} \mathbb{E} \mathbb{E} [ | W_t - W |^2  \mathbb{I}( | W_t - W |^2 > \rho ) | W]\ .$$
Since M is compact, we can find a constant $C > 0$ such that the first derivative of f is bounded by C. Then
$$| f(X_t) - f( X) | \leq C \cdot |X_t - X|\ .$$
Choose a radius $r$ such that $r \leq \frac{\rho}{C}$, then we have the following upper bound:
\begin{align} \nonumber
\lim_{t \rightarrow 0} \frac{1}{t} \mathbb{E} [ | W_t - W |^2  \mathbb{I}( | W_t - W |^2 > \rho ) ] & \leq \lim_{t \rightarrow 0}\frac{1}{t} \mathbb{E} \mathbb{E} [ (f(X_t) - f(X))^2 \mathbb{I}(  \tau_{\sqrt{\rho}} \leq t) | X] \\  \nonumber
& \leq \lim_{t \rightarrow 0} 4 \cdot \norm{f}^2_{max} \frac{C_r (x)}{t^{d/2}} e^{-r^2 / 2t}\\ \nonumber
& = 0\ ,  \nonumber
\end{align}
where we use lemma 3.4 in the penultimate step.
\item $\nabla H = 0$, $\Delta_w = \Delta + \nabla H\ .$\\
When $ H(x) \neq constant$, there is a nonzero drift term in the SDE:
$$d U_t = \sum_{i=1}^d \mathcal H^i ( U_t ) \circ d B^i_t  - \frac{1}{2} \widetilde{\nabla H} (U_t) dt \ .$$
By definition, the anti-development of $U_t$ can be written as $B_t + V_t$, where $V_t$ satisfies the following SDE:   
$$dU_t = \sum_{i=1}^d \mathcal H^i ( U_t ) \circ ( d B^i_t + \dot{V}(t) dt )\ .$$
Since the vector field $\widetilde{\nabla H(x)}$ is bounded, we know that the $|\dot{V(t)}|$ is also bounded and $V(t) = \int_0^t \dot{V_s} ds$ belongs to the Cameron-Martin space. Define $M_t$ as
$$M_t = \exp( - \int_0^t \dot{V_s} dB_s - \frac{1}{2} \int_0^t |\dot{V_s}|^2 ds ) \ .$$
Since $\mathbb{E} [ e^{\frac{1}{2} \int_0^t |\dot{U_s}|^2 ds} ] < \infty$, by Novikov's criterion,
$M_t$ is a uniformly integrable martingale and Girsanov's transform applies. We denote the original probabilistic measure as $P$ and define the new measure $Q$ as
$$dQ = M_t dP\ ,$$
then the $\mathbb{R}^d$- valued process $ B_t + V_t$ is a Brownian motion under measure $Q$. \\This implies that the development of $B_t + V_t$ , which is exactly $U_t$, becomes a Brownian motion on $\mathcal O (M)$ under the new measure $Q$. This shows that $X_t = \pi(U_t)$ is a Brownian motion on $M$. Let $\tau^X_r$ be the first exit time of the ball $B(x , r)$ of $X_t$ under the original measure $P$, then
$$ \mathbb{E} \{ \mathbb{I}( \tau^X_r \leq t ) \} = \mathbb{E}_Q \{ \mathbb{I}( \tau^X_r \leq t ) \frac{dP}{dQ} \}\ .$$
lemma 3.4 implies:
\begin{align} \nonumber
\mathbb{E}_Q \{ \mathbb{I}( \tau^X_r \leq t ) \frac{dP}{dQ} \} &= \mathbb{E}_Q \{ \mathbb{I}( \tau_r \leq t ) \frac{dP}{dQ} \}\\ \nonumber
& \leq \sqrt{\mathbb{E}_Q \{ \mathbb{I}( \tau_r \leq t ) \} \cdot \mathbb{E}_Q (\frac{dP}{dQ})^2}\ . \nonumber
\end{align}
Since $\mathbb{E}_Q (\frac{dP}{dQ})^2$ is bounded, we can find a $ C' > 0$ such that $\mathbb{E}_Q (\frac{dP}{dQ})^2 \leq C'$. Therefore,
$$\mathbb{E} \{ \mathbb{I}( \tau^X_r \leq t ) \} \leq C' \cdot \sqrt{\mathbb{P}( \tau_r \leq t)}\ .$$
This implies that $\mathbb{E} \{ \mathbb{I}( \tau^X_r \leq t ) \}$ also has the exponential decay as $t \rightarrow 0$ and the rest of the proof goes as in the Brownian motion case.
\end{enumerate}
\end{proof} 
Since we have checked all the conditions of Theorem 3.1, we get the following conclusion:

\begin{thm}\label{eigenfunctions}
Let $\{ f_i (x) \}$ be an orthonormal (with respect to $d\mu (x)$) sequence of  eigenfunctions of $\Delta_w $ with corresponding eigenvalues $- \nu_i$.  Let $X$ be a random point of $M$ sampled from measure $d\mu$. Then if $W$ is the random vector: $(f_1(X),
\ldots,f_k(X))$, 
$$d_W(W,Z)\le\left[\max_{1\le i\le k} (\mu_i^{-1})\right] \cdot
\mathbb{E} \sqrt{\sum_{i,j=1}^k [ \langle \nabla f_i(X) , \nabla  f_j(X) \rangle - \mathbb{E} \langle \nabla f_i(X) , \nabla  f_j(X) \rangle ]}\ .$$
\end{thm}
\section{Linear spherical symmetric statistics}
It requires further information to bound the $\mathbb{E} | \norm{\nabla f(X)}^2 - \mathbb{E} \norm{\nabla f(X)}^2 |$ term to get a central limit theorem out of Theorem 6. In this section, we turn from a general manifold to specific models where we can do precise calculation. As a warm-up example, we recover a corollary in \cite{meckes-thesis}. For one dimensional normal distribution, the stein equation is a first order differential equation:
$$f'(x) - xf(x) = g(x) - \mathbb{E}g(Z)\ .$$
Since the Stein operator is first order for the univariate normal distribution, the corresponding metric is the total variation distance. In \cite{meckes-thesis}, the author proved the following abstract approximation theorem:
\begin{thm}\label{abscont}Suppose that $(W,W_t)$ is a family of 
exchangeable pairs defined on a common probability space with 
$\E W=0$ and $\E W^2=1$.  Suppose there is a random variables $E=
E(W)$ and  a constant $\lambda$
such that 
\begin{enumerate}
\item $$\E\left[W_t -W\big|
W\right]=-\lambda t W+O(t^2),$$\label{lindiff3}
\item $$\E\left[(W_t-W)^2\big|
W\right]=2\lambda t +E + O(t^2),$$\label{quaddiff3}
\item For each $ \rho > 0$,$$\lim_{t \rightarrow 0} \frac{1}{t} \E [ | W_t - W |^2  \mathbb{I}( | W_t - W |^2 > \rho ) ] = 0 \ . $$\label{cubediff3}
\end{enumerate}
Then
$$d_{TV}(W,Z)\le \frac{1}{\lambda} \mathbb{E} \left|E\right|\ ,$$
where $Z\sim N(0,1).$ 
\end{thm}
We restate a corollary of  Theorem 2.2 in \cite{meckes-thesis}  with a slightly better constant:
\begin{coro}
Let $X$ be a uniform random point on the scaled n dimensional sphere $\sqrt{n} S^{n-1}$, the sphere in $\mathbb{R}^n$ of radius $\sqrt{n}$. For $\theta \in S^{n-1}$, define $W = \inprod {X}{\theta}$. Then
$$d_{TV}(W,Z)\le  \frac{2\sqrt{2}}{\sqrt{(n-1)(n+2)}}\ ,$$
where Z is standard Gaussian $N(0,1)$.
\end{coro}
\begin{remark}
By spherical symmetry, it suffices to prove for the case $\theta = e_1$, where $e_1$ is the direction vector of the first coordinate. In the original paper \cite{meckes-thesis},  the exchangeable pair was created by rotating $W$  in a random two-dimensional subspace through a deterministic angle $\sin^{-1} (\epsilon)$. \\
\end{remark}
Let $X_t$ be the Brownian motion on sphere starting at $X$. Since $X$ is sampled by the uniform measure on the sphere, which is the stationary distribution of the Brownian motion. Denote $W_t$, $W$ as the first coordinate component of $X_t$ and $X$. Then it's easy to verify that $( W_t, W)$ forms a family of exchangeable pairs. The corresponding generator is $\Delta_{\sqrt{n} S^{n-1}}$ with no drift term. By lemma 3.2, we have
\begin{align} \nonumber
\mathbb{E} [ W_t - W | W] & = \Delta_{\sqrt{n} S^{n-1}} W + O( t^2)\\  \nonumber
& = \frac{1}{n} \Delta_{S^{n-1}} W + O(t^2) \\ \nonumber
& =  \frac{1}{n} r^2 [ \Delta_{\mathbb{R}^n} W - \frac{\partial^2}{\partial^2 r } W - \frac{n - 1}{r} \frac{\partial}{\partial r} W ]\\ \nonumber
& = - \frac{n-1}{n} W\ . \nonumber
\end{align}
By lemma 3.3, we have
$$\mathbb{E} [ (W_t - W)^2 | W ] = t \norm{\nabla X_1}^2 + O( t^2)\ .$$
Let $E =  \norm{\nabla X_1}^2 - \mathbb{E} \norm{\nabla X_1}^2 $, then the only thing left is to bound the variance of $E$. We need the following proposition for integrating polynomials over spheres.
\begin{prop}
Let $P(x)=|x_1|^{\alpha_1}$.  
Then if $X$ is uniformly distributed on $\sqrt{n} S^{n-1}$, 
$$\mathbb{E}\big[P(X)\big]=\frac{\Gamma(\beta_1)\Gamma(\frac{n}{2}) n^{\frac{1}{2} \alpha_1} (\Gamma(\frac{1}{2})^{n-1}}{\Gamma(\beta_1+\frac{n-1}{2})\pi^{n/2}},$$
where $\beta_1=\frac{1}{2}(\alpha_1+1)$ and 
$$\Gamma(t)=\int_0^\infty s^{t-1}e^{-s}ds=2\int_0^\infty r^{2t-1}e^{-r^2}dr.$$
When $\alpha_1 = 4$, we have
$$\mathbb{E}\big[x_1^4] = \frac{3n}{n+2}\ .$$
\end{prop}
Now we are ready to calculate the variance:
\begin{lem}Let $E$ be defined as above. Then
$$Var [E]= \frac{2(n-1)}{n^2 (n+2)} \ .$$
\end{lem}
\begin{proof}
\begin{align}  \nonumber
Var [E] & = \mathbb{E} || \nabla X_1 ||^4 - ( \mathbb{E} || \nabla X_1 ||^2 )^2\\ 
& = \mathbb{E} || \nabla X_1 ||^4 - \frac{(n-1)^2}{n^2}\ .        
\end{align}
Note that the covariant derivative on $\sqrt{n} S^{n-1}$ is the projection of the $\mathbb{R}^n$- gradient to the tangent space, which is the hyperplane orthogonal to the radial vector $\frac{\vec{x}}{||\vec{x}||}$. We have 
\begin{align}  \nonumber
|| \nabla X_1 ||^2 &= <  \nabla X_1,  \nabla X_1> \\  \nonumber
& = < \nabla_{\mathbb{R}^n} X_1 - (X \cdot\nabla_{\mathbb{R}^n} X_1)\frac{\vec{x}}{||\vec{x}||},  \nabla_{\mathbb{R}^n} X_1 - (X \cdot\nabla_{\mathbb{R}^n} X_1) \frac{\vec{x}}{||\vec{x}||}> \\  \nonumber
& = ||\nabla_{\mathbb{R}^n} X_1 ||^2 - \frac{1}{n} (X \cdot\nabla_{\mathbb{R}^n} X_1)^2\\  \nonumber
& = 1 - \frac{1}{n} X_1^2\ . \nonumber
\end{align}
From proposition 8, 
$$\mathbb{E}[X_1^2] =1,\ \ \ \  \mathbb{E}[X_1^4] = \frac{3n}{n+2}$$
This implies
\begin{align}  \nonumber
\mathbb{E}|| \nabla X_1 ||^4 & = 1 - \frac{2}{n} \mathbb{E} X_1^2 + \frac{1}{n^2} \mathbb{E} X_1^4 \\  \nonumber
& = 1 - \frac{2}{n} + \frac{3}{n(n+2)}\ .  \nonumber
\end{align}
Plugging in (4.1), we get
\begin{align}  \nonumber
Var [E] & =  \mathbb{E} || \nabla X_1 ||^4 - \frac{(n-1)^2}{n^2}\\  \nonumber
& = \frac{2(n-1)}{n^2 (n+2)}\ .  \nonumber
\end{align}
\end{proof}
Since $X_1$ is an eigenfunction of $\Delta_{\sqrt{n} S^{n-1}}$ with eigenvalue $- \frac{n -1 }{n}$, we can apply Theorem 4.1:
\begin{align}  \nonumber
d_{TV}(W,Z) & \leq \frac{2n}{n-1}\mathbb{E} |E| \\  \nonumber
& \leq  \frac{2n}{n-1} \sqrt{Var [E]}\\   \nonumber
& \leq \frac{2n}{n-1} \sqrt{\frac{2(n-1)}{n^2 (n+2)}}\\   \nonumber
& \leq \frac{2\sqrt{2}}{\sqrt{(n-1)(n+2)}}\ .   \nonumber
\end{align}
\begin{remark}
In \cite{meckes-thesis}, $W_{\epsilon}$ was a rotation of $W$ by a deterministic angle, it's obvious that the $\mathbb{E}  | W_\epsilon - W|^3 $ is of order $O(\epsilon^3)$. On the other hand, our perturbation is by Brownian motion and the corresponding generator $\Delta$ is rotational invariant, so the spherical symmetry is preserved. Although we can calculate the third moment by proposition 4.3 in this case, it's easier to verify the $\lim_{t \rightarrow 0} \frac{1}{t} \E [ | W_t - W |^2  \mathbb{I}( | W_t - W |^2 > \rho ) ] = 0  $ condition rather than bounding the higher moments in more complex cases.
\end{remark}
\section{Approximate exponentiality of $|\Tr(U^k)|^2$}
Let $U_n$ denotes the unitary group of dimension $n \times n$. In \cite{hr}, the authors investigated the linear statistics of unitary group. We briefly review the relevant background. Let $g \in L^1 (T)$ be real valued function on the one dimensional torus. We can view $g$ as a $2\pi$-periodic function on $\mathbb{R}$ by identifying $g(e^{ix})$ with $g(x)$. Then
$$ \Tr g(U) := \sum_{j =1}^{j=n} g(\theta_j) \ ,$$
where $e^{i \theta_j},\ \ 1 \leq j \leq n$ are the eigenvalues of $U_n$. We can expand $g$ into Fourier series: $g(\theta) = \sum_{j \in \mathbb{Z} } g_k e^{ik\theta}$, then we have
$$\Tr g(U) =  \sum_{j \in \mathbb{Z}} g_k \Tr (U^k)\ .$$

In this  way, the randomness is concentrated in the $\Tr (U^k )= \sum_{j =1}^{n} e^{ik\theta_j} $ for each frequency $k$, which can be seen as random Fourier basis functions. Thus it's worth studying how the norm of those basis functions behave quantitatively as the dimension $n \rightarrow \infty$. In \cite{jn}, the authors showed that as n tends to infinity, the distribution of $|\Tr(U_n)|^2$ tends to an exponential distribution with mean equals to 1. In this section, we will show that for a general power k, $|\Tr (U^k)|^2$ tends to an exponential distribution by an 
Infinitesimal version of stein method for exponential distribution. Moreover, the setting here can be generalized to general circular ensemble without any change, which will be the main topic of next section.
\subsection{Abstract approximation theorem for exponential distribution}
The first step towards establishing the stein method for exponential distribution is to find an operator $\mathcal{A}$, which acts on a large enough class of functions and 'characterizes' the exponential distribution in the sense that
$$ \mathcal{A}f(Z)=0\ ,$$
for all $f$ in a large enough class of functions and $Z$ follows the exponential distribution of mean 1. It turns out the characterizing operator for the exponential distribution is of first order:
$$\mathcal{A}f(x) = xf' (x) - (x - 1)f(x)\ .$$
It may be worth pointing out that the $-(x - 1)$ part of the Stein operator has a connection with the scale score function (see \cite{cl}) of the exponential distribution:
$$ \frac{d}{d \lambda} \log ( \lambda \exp{- \lambda x}) = \frac{1}{\lambda} - x\ .$$

For functions $h$ in the class of interest, define $f_h$ to be the solution of the Stein equation
$$
\mathcal{A}f_h(x)=h(x)-  \mathbb{E}h(Z).
$$
An iterative technique for bounding derivatives of solutions of Stein equations has been developed. However, for the infinitesimal version of exchangeable pair, we only need to bound the first derivative and the higher order term is automatically cut off. We restate the following fundamental lemma 2.2 in \cite{jn}:
\begin{lem}
Let $Z$ be a mean one exponential random variable.
If $h$ is a function such that the following integrals are well defined,
then
$$
f(w) =f_h(w)=-\frac{e^w}{w}\int_w^\infty (h(x)-\mathbb{E} h(Z))e^{-x}dx \label{stnsol}
$$
solves the differential equation
$$
wf'(w)-(w-1)f(w)=h(w)-\mathbb{E} h(Z). \label{stneq}
$$
If $h$ is absolutely continuous with $\norm{h'}<\infty$, then
$$
\norm{f}\leq \left(1+\frac{2}{e}\right)\norm{h'}, \hspace{5mm} \norm{f'}\leq  2\norm{h'}.        \label{abcobd}
$$
\end{lem}
Note that in our case, since we don't have a bound of the form $\norm{f'} \leq C \norm{h}$, we won't expect to have an approximation measured by the total variation distance. Instead, we need to introduce the following class of functions for $t,x \geq 0$, and $\delta \geq 0$,
 \begin{equation}\label{htd}
 h_{t,\delta}(x) = \left\{
     \begin{array}{ll}      
       1, &  x \leq t-\delta,\\
       1-\frac{2(x-t+\delta)^2}{\delta^2}, &t-\delta<x\leq t-\delta/2 ,\\
       \frac{2(x-t)^2}{\delta^2},&t-\delta/2<x\leq t ,\\
       0, & x > t .              
     \end{array}
   \right.
\end{equation}
We can derive the following bound on the Kolmogolov distance (see lemma 2.3 in \cite{jn}):
\begin{lem}\label{lemhtd}
If $t\geq0$, $\delta>0$, and  $h_{t,\delta}$ is defined by (5.1), then
$$
\norm{h_{t,\delta}}=1, \hspace{5mm} \norm{h_{t,\delta}'}=2/\delta, \hspace{5mm} \norm{h_{t,\delta}''}=4/\delta^2.
$$
If $W\geq0$ is a random variable and $Z$ has the exponential distribution with mean one, then
$$
d_K(W,Z)\leq\sup_{t\geq0}\abs{\mathbb{E} h_{t,\delta}(W)-\mathbb{E} h_{t,\delta}(Z)}+\delta/2. \label{smo}
$$
\end{lem}
Now we are ready to prove the infinitesimal version of approximation theorem for exponential distribution with mean equals to 1:
\begin{thm}\label{abscont}Let Z be a mean 1 exponential distributed random variable. If $W \geq 0$ is random variable with finite second moment and $(W,W_t)$ is a family of 
exchangeable pairs defined on a common probability space. Let $F$ be a $\sigma -$algebra s.t $\sigma (W) \subseteq F$ and Suppose there are $F-$measurable random variables $E$ and $E'$ such that
\begin{enumerate}
\item $$\mathbb{E}\left[W_t -W\big|
F\right]=\Lambda t(1-W)+tE + O(t^2),$$\label{lindiff3}
\item $$\mathbb{E}\left[(W_t-W)^2\big|
F\right]=2\Lambda t W +tE'+ O(t^2),$$\label{quaddiff3}
\item For each $ \rho > 0$,$$\lim_{t \rightarrow 0} \frac{1}{t} \mathbb{E} [ | W_t - W |^2  \mathbb{I}( | W_t - W |^2 > \rho ) ] = 0 \ . $$\label{cubediff3}
\end{enumerate}
Then, for $\forall \delta \geq 0$,
$$d_{K}(W,Z)\le \frac{1}{\Lambda \delta}\left[2
\mathbb{E}\left|E'\right|+(1+\frac{2}{e})\mathbb{E}\left|E\right|\right] + \frac{\delta}{2}\ .     \eqno{(5.2)}$$

\end{thm}
\begin{proof}
Since $(W_t, W)$ is an exchangeable pair,  we have the identity         
\begin{align}        \nonumber
0 &= \mathbb{E} [ (W_t - W) (f(W_t) + f(W))]\\            \nonumber
& = \mathbb{E} [ (W_t - W) (f(W_t) - f(W)) + 2(W_t - W) f(W)] \ ,           \nonumber
\end{align}
by Taylor expansion,
$$0 = \mathbb{E} \{ \mathbb{E} [ (W_t - W)^2 | W] f' (W) + 2 \mathbb{E} [(W_t - W)| W] f(W)] + R\}\ .$$
For the reminder term $R$, there exists a real number $K$ depending on the function $f$, such that
$$|R| \leq K |W_t - W|^3\ .$$
Fix $ \rho > 0$, decomposing the integrand into two sets
$$\{|W_t - W| \leq \rho \} \ \  and \ \ \{|W_t - W| \ge \rho \},$$
we obtain
\begin{align}           \nonumber
\lim_{t \rightarrow 0} \frac{1}{t} \mathbb{E} |R| &\leq \lim_{t \rightarrow 0} \frac{K}{t} \mathbb{E} [ |W_t - W|^3  \mathbb{I}( |W_t - W| \leq \rho ) + |W_t - W|^3  \mathbb{I}( |W_t - W| > \rho )]\\       \nonumber
& \leq \lim_{t \rightarrow 0}  K \rho \frac{\mathbb{E} |W_t - W|^2}{t} + \lim_{t \rightarrow 0} \frac{K}{t} \mathbb{E} [ |W_t - W|^2 \mathbb{I}( |W_t - W| > \rho )]  \ ,      \nonumber
\end{align}
by condition (3), the second term is zero and by condition (2),
$$ \lim_{t \rightarrow 0} \frac{K \rho}{t} \mathbb{E} | W_t - W|^2 \leq cK\rho$$
for some constant c that depends on the distribution of $W$. Let $\rho \rightarrow 0$, 
$$\lim_{t \rightarrow 0} \frac{1}{t} \mathbb{E}|R| = 0\ .$$
Now we are left with two terms, dividing both side by t and take the limit, we have
$$0 = \mathbb{E} [ 2 \Lambda W f'(W) + E' f'(W) + 2 \Lambda (1-W) f(W) + E f(W) ]\ .$$
Since $f$ is the solution of the stein equation with respect to $g(x)$, this implies
$$0 = 2 \Lambda \mathbb{E} [ g(W) - \mathbb{E} g(Z) ] + \mathbb{E} [ E' f'(W) + Ef(W)]\ ,$$
by lemma 5.1, we have
$$\mathbb{E} [ g(W) - \mathbb{E} g(Z) ] \leq \frac{1}{2 \Lambda} ( 1+ \frac{2}{e}) \norm{g'} \cdot \mathbb{E} |E| + \frac{1}{\Lambda} \norm{g'} \mathbb{E} |E'|\ .$$
Let $g(x)$ be in the function class of $h_{t,\delta}$, then (5.2) leads to
\begin{align}           \nonumber
d_K(W,Z)&\leq\sup_{t\geq0}\abs{\mathbb{E} h_{t,\delta}(W)-\mathbb{E} h_{t,\delta}(Z)}+\delta/2 \\         \nonumber
&\leq \frac{1}{\Lambda \delta} (1+ \frac{2}{e}) \mathbb{E} |E| + \frac{2}{\Lambda \delta} \mathbb{E} |E'| +\delta/2    \ .         \nonumber
\end{align}
\end{proof}
\subsection{Approximate normality of  $|\Tr (U^k)|^2$}
To apply Theorem 5.3, the following lemma on the moment calculation of traces of unitary group is crucial (see \cite{hr,ds}):
\begin{lem}  Let $U$ be Haar distributed on $U(n,\mathbb{C})$. Let $(a_1,\cdots,a_k)$ and
$(b_1,\cdots,b_k)$ be vectors of non-negative integers. Then one has that for all $n \geq \sum_{i=1}^k (a_i+b_i),$
\[ \left[ \prod_{j=1}^k Tr(U^j)^{a_j} \overline{Tr(U^j)^{b_j}} \right] = \delta_{\vec{a} \vec{b}} \prod_{j=1}^k j^{a_j} a_j!. \]
\end{lem}
Let $\lambda = (\lambda_1, \lambda_2, \dots )$ be a partition, that is, the sequence is in non-increasing order and only finite of $\lambda_i$ are nonzero. Let $m_j$ denotes the multiplicity of part j in $\lambda$ and $l(\lambda)$ the length of $\lambda$:\ $l(\lambda) = m_1(\lambda) + m_2(\lambda) + \dots,$ then $$p_{\lambda} (U) : = \prod_{i = 1}^{l(\lambda)} \Tr (U^i )^{m_i} \ .$$  
We need a lemma from \cite{Lt} for the expression of the gradient and also the Laplacian of the trace powers of unitary group in terms of $p_{\lambda}$:
\begin{lem}  Let $j > 0$ an integer, $p_j(U(n))$ be as above. Then $\overline{p_j} = p_{-j}$ and
\begin{enumerate}
\item \[ \Delta_{U(n)} p_j = - nj p_j - j \sum_{1 \leq l <j} p_{l,j-l}.\]
\item \[ \Delta_{U(n)} p_{j,j} = - 2nj p_{j,j} - 2j^2 p_{2j} - 2j p_j \sum_{1 \leq l <j} p_{l,j-l} .\]
\item \[ \Delta_{U(n)} \left( p_j \overline{p_j} \right) = 2j^2n - 2njp_j \overline{p_j} - jp_j \sum_{1 \leq l < j} \overline{p_{l,j-l}} - j \overline{p_j} \sum_{1 \leq l < j} p_{l,j-l} .\]
\end{enumerate}
\end{lem}
Form this we can easily deduce the formula for the gradient of trace power functions:
\begin{align}     \nonumber
(\nabla p_k(U)) \cdot (\nabla p_k(U)) &= \frac{1}{2}[ \Delta( p_k(U)^2) - 2 p_k(U) \Delta  p_k(U)]\\         \nonumber
& = -k^2 \cdot  p_{2k}(U) \ .       \nonumber
\end{align}

Let $W(U) = |\Tr (U^k)|^2 / k = \Tr (U^k)\cdot \overline{\Tr (U^k)} / k$. We plan to prove that $W$ tends to the exponential distribution by checking the three conditions and identifying $\Lambda$, $E$ and $E'$ of Theorem 5.3. 

The first step is to perturb $W$ by running a Brownian motion starting at $U$ and we denote it by $U_t$. Let $W_t = |\Tr (U_t^k)|^2 / k$, then $(W_t , W)$ forms an exchangeable pair under the Haar measure of the unitary group. Since $ W_t = e^{\Delta t} W$, we have
\begin{align}        \nonumber
\mathbb{E} [ W_t - W | W] & = t \Delta (p_k \overline{p_k} / k) + O(t^2)\\            \nonumber
& = 2n( k - p_k \overline{p_k} ) - p_k \sum_{1 \leq l < k} \overline{p_{l,k-l}} - \overline{p_k} \sum_{1 \leq l < k} p_{l,k-l}\\         \nonumber
&= 2nk( 1 -W) - p_k \sum_{1 \leq l < k} \overline{p_{l,k-l}} - \overline{p_k} \sum_{1 \leq l < k} p_{l,k-l} \ .    \nonumber
\end{align}
Unlike the previous case, the residue term $E : = p_k \sum_{1 \leq l < k} \overline{p_{l,k-l}} + \overline{p_k} \sum_{1 \leq l < k} p_{l,k-l}$ is not zero.\\
By lemma 3.3, we have
\begin{align}     \nonumber
\mathbb{E} [ (W_t - W)^2 | W] & = 2 t \nabla W \cdot \nabla W + O(t^2)\\          \nonumber
&= 2t \nabla p_k \cdot \nabla p_k \cdot (\overline{p_k} \cdot \overline{p_k})/ k^2                   \nonumber
 + 4t  \nabla p_k \cdot \nabla \overline{p_k} \cdot ( p_k \cdot  \overline{p_k}) / k^2 \\       \nonumber
&\ \   +  2t \nabla \overline{p_k} \cdot \nabla \overline{p_k} \cdot (p_k \cdot p_k) / k^2 +   O(t^2) \\       \nonumber
& = -2 p_{2k} (\overline{p_k})^2 + 4 n p_k \overline{p_k} - 2 \overline{p_{2k}} (p_{2k})^2 + O(t^2)\\    \nonumber
& = 4knW - 2p_{2k} (\overline{p_k})^2 - 2 \overline{p_{2k}} (p_k )^2 + O(t^2)  \ .   \nonumber
\end{align}
We find out that $\Lambda = 2kn$, $E' =  - 2p_{2k} (\overline{p_k})^2 - 2 \overline{p_{2k}} (p_k )^2$.  The next task is to derive upper bounds for the norm of the two residue terms $E$ and $E'$.
\begin{align}    \nonumber
\mathbb{E} [ E^2] &= \mathbb{E} [p_k^2 ( \sum_{1 \leq l < k} \overline{p_{l,k-l}})^2] + \mathbb{E}[\overline{p_k}^2 ( \sum_{1 \leq l < k} p_{l,k-l})^2] \\    \nonumber
&\ \ \ \ + 2 p_k \overline{p_k} ( \sum_{1 \leq l < k} \overline{p_{l,k-l}} \cdot \sum_{1 \leq l < k} p_{l,k-l}) \ ,    \nonumber
\end{align} 
by lemma 5.5 and only keep terms with non-zero expectation, we have

\begin{align} \nonumber
\mathbb{E} [ E^2] &= \mathbb{E} \sum_{1 \leq l < k} (p_{l,k -l} \cdot p_k) \cdot \sum_{1 \leq l < k} (\overline{p_{l,k - l}} \cdot \overline{p_k}) \\     \nonumber
&  =
\left\{
             \begin{array}{lr}
              \sum_{1 \leq l < k} l \cdot (k - l) = \frac{k^3 - k}{6}, \ \ \ \ \  \ \ \ \ \ \ \ \ \ \ \ \ if\  k\  is \ odd&  \\     
              \sum_{1 \leq l < k}  l \cdot (k - l)  + \frac{k^2}{4} = \frac{2k^3 + 3k^2 - 2k}{12}.  \ \ if\  k\  is \ even &  
             \end{array}
\right.
\end{align}
It implies that for all $k > 0$,
$$\mathbb{E} [ E^2] \leq \frac{2k^3 + 3k^2 - 2k}{12}\ ,$$
and
\begin{align}   \nonumber
\mathbb{E} [ (E')^2] &= 4 \mathbb{E} [ (p_{2k})^2 (\overline{p_k})^4] + 4 \mathbb{E} [ (\overline{p_{2k}})^2 (p_k)^4]\\        \nonumber
&\ \ \ + 8 \mathbb{E} [p_{2k} (p_k)^2 \cdot \overline{p_{2k}} (\overline{p_{k}})^2] \ .         \nonumber
\end{align}
By lemma 5.5, the first two terms vanish, we have
\begin{align}      \nonumber
\mathbb{E} [ (E')^2] &=8 \mathbb{E} [p_{2k} \overline{p_{2k}} \cdot (p_k \overline{p_{k}})^2] \\      \nonumber
& = 32 k^3\ .        \nonumber
\end{align}
Finally, by Theorem 5.3, we conclude that for $\forall \delta > 0$,
\begin{align}          \nonumber
d_{K}(W,Z) &\le \frac{1}{\Lambda \delta}\left[2
\mathbb{E}\left|E'\right|+(1+\frac{2}{e})\mathbb{E}\left|E\right|\right] + \frac{\delta}{2} \\       \nonumber
&\leq \frac{1}{\Lambda \delta}\left[2
\sqrt{\mathbb{E}\left|E'\right|^2}+(1+\frac{2}{e})\sqrt{\mathbb{E}\left|E\right|^2}\right] + \frac{\delta}{2}\\           \nonumber
&\leq \frac{1}{2kn\delta}(1+\frac{2}{e})\sqrt{\frac{2k^3 + 3k^2 - 2k}{12}} + \frac{1}{kn\delta}\sqrt{32k^3} + \frac{\delta}{2}\\         \nonumber
& \leq (\frac{1}{2} + 4\sqrt{2})\frac{\sqrt{k}}{n\delta} + \frac{\delta}{2}\ .
\end{align}
Choose the optimal $\delta = \sqrt{(1+8\sqrt{2})\frac{\sqrt{k}}{n}}$, we have proved the following theorem:
\begin{theorem} \label{main2} Let $W=|Tr(U^k)|^2$, where $U$ is from the Haar measure of unitary group and $k$ is a positive integer. Let $Z$ denotes the exponential  random variable with mean one. Then we have
$$d_{K}(W,Z) \le \sqrt{(1+8\sqrt{2})\frac{\sqrt{k}}{n}}\ .$$
\end{theorem}
\section{Exponentiality of circular ensemble}
As we all know, the trace power functions on unitary matrix group are actually functions of the eigenvalues. From this point of view, it's crucial to study the distribution of eigenvalues induced by the Haar measure of the Lie group. It turns out that the distribution of eigenvalues has the following form:
$$p( \theta_1, \cdots , \theta_n ) = \frac{1}{Z_{n,\beta}} \prod_{1 \leq k < j \leq n} | e^{i \theta_k} - e^{i \theta_j }|^{\beta}\ ,$$
where $ ( \theta_1, \cdots , \theta_n ) \in \mathbb{R}^n_{[0,2\pi ]}$ and $Z_{n,\beta}$ is the normalization constant. In the unitary matrix case, $\beta = 2$. The circular orthogonal ensemble ($\beta = 1$) and the circular symplectic ensemble ($\beta = 4$) also belong to this circular ensemble family. 

In this section, we generalize the results of section 5 to the circular ensemble for all $\beta \ge  1$. Since there is no underlying Lie group for general $\beta$, we cannot use the Brownian motion on the Lie group to characterize it. Instead, we need to find a diffusion process that models the circular ensemble directly. The potential function of circular ensemble is
$$H(\theta_1, \cdots , \theta_n ) = \beta \ln \prod_{1 \leq k < j \leq n} | e^{i \theta_k} - e^{i \theta_j }|\ .$$
Let's order $( \theta_1, \cdots , \theta_n )$ such that $\theta_1 \leq \cdots \leq \theta_n$ in the compact set $\mathbb{R}^n_{[0,2\pi ]}$, then the corresponding drift vector field is
\begin{align}   \nonumber
\nabla H &= \nabla_i [ \sum_{ 1 \leq k < j \leq n} \beta \ln | e^{i \theta_k} - e^{i \theta_j} | ]\frac{\partial}{\partial \theta_i}  \\
& = \nabla_i [ \sum_{ 1 \leq k < i} \beta \ln \left (2 \sin \frac{\theta_k - \theta_i}{2}\right ) + \sum_{ i < k \leq n} \beta \ln \left (2 \sin \frac{\theta_i - \theta_k}{2}\right )  ]\frac{\partial}{\partial \theta_i}\\          \nonumber
& = \frac{\beta}{2} \sum_{k \neq i} \cot \left ( \frac{\theta_i - \theta_k }{2} \right ) \ .
\end{align}
So by our formal argument, the diffusion process for a general $\beta$ ensemble satisfies the following stochastic differential equation:
$$d X^j_t = \sqrt{2} d B^j_t + \frac{\beta}{2} \sum_{ 1 \leq k \neq  j \leq n} \cot \left ( \frac{\theta_j - \theta_k }{2} \right ) dt , \ j = 1,2, \cdots , n.$$
This is exactly the circular Dyson Brownian motion (CDBM). The corresponding Witten Laplacian is $\Delta + \frac{\beta}{2} \sum_{ 1 \leq k \neq  j \leq n} \cot \left ( \frac{\theta_j - \theta_k }{2} \right ) \frac{\partial}{\partial \theta_j}$, which is the so-called Dyson operator $D$. In \cite{Webb}, the author has applies CDBM to prove a multidimensional CLT for the circular ensemble. Note that CDBM is well defined in the open simplex
$$\Sigma_{n} := \{ X = (x_1, \dots , x_n ) \in \mathbb{R}^n_{[0,2\pi ]} \ : x_1 < \cdots < x_n \}$$
and $\partial \Sigma_{n}$ is the corresponding boundary set. We summarize the results we need in \cite{cl} (see also \cite{lh} ) as the following lemma:
\begin{lem} Let $\beta \ge 1$ and suppose that the initial data $X(0) \in \Sigma_{n}$. Then there exists a unique solution to (1) in the space of continuous functions $(X(t))_{t \ge 0} \in C( \mathbb{R}_{+} , \Sigma_{n})$. Moreover, the circular ensemble is the unique equilibrium of CDBM and CDBM is reversible with respect to this distribution.
\end{lem}

Now, let's define the analogue of $\Tr (U^k)$ for general $\beta$ ensemble and $k \in \mathbb{Z}$:
$$p_k (x) = \sum_{j =1}^n e^{ik x_j }, \ \ \ \  x = (x_1,\dots,x_n) \in \Sigma_{n}\ .$$
Note that this is a re-expression of $\Tr (U^k)$ if we set $(e^{ix_1}, \dots , e^{ix_n})$ as the ordered sequence of eigenvalues of unitary matrix and $\overline{ p_k} = p_{-k}$ as before. To state the main theorem of this section, we need to introduce some constants. Set $\alpha = 2/ \beta$, then $C_E$ ,$ C_{E'}$ is defined by
$$C_E = \max \{ |A -1| , |B-1| \},\ \ \ \ C_{E'} = \max \{ |A' -1| , |B'-1| \},$$
where
$$A = (1 - \frac{|\alpha - 1|}{n-2k + \alpha})^{2k},\ \ \ \ B = (1 + \frac{|\alpha - 1|}{n-2k + \alpha})^{2k},$$
and
$$A' = (1 - \frac{|\alpha - 1|}{n-4k + \alpha})^{4k},\ \ \ \ B' = (1 + \frac{|\alpha - 1|}{n-4k + \alpha})^{4k}\ .$$
As in the last section, let Z be an exponential random variable with mean 1. Then we have the following theorem:
 \begin{thm} Let $W= \frac{\beta}{2k} p_k (X) \overline{p_k}(X)$, where X is a random point sampled by the $\beta$ circular ensemble in $\Sigma_n$. Then there exists an integer $N(k,\beta)$ which depends on $k$ and $\beta$, such that for $n \ge N((k,\beta)$, the Kolmogorov distance between W and Z is bounded by:
$$2 \sqrt{\frac{1}{\sqrt{\beta} n } \sqrt{80 C_{E'} k} + \frac{1}{ \beta n} (1+ \frac{2}{e}) \sqrt{2 C_{E'} k^3}}\ .$$
\end{thm}
\begin{rem}
We need to let $n \ge N(k,\beta)$ to make the constants $A$ and $A'$ positive and less than one. This will be important in lemma 6.2 and lemma 6.3.
\end{rem}
We denote the CDBM starting from the random point X at time t as $X_t$ and let $W_t =\frac{\beta}{2k} p_k (X_t) \overline{p_k}(X_t)$. By the argument in section 2, we know that $(W,W_t)$ forms a continuous family of exchangeable pairs. To apply the Stein method, we need to generalize formulas in lemma 5.5 to general $\beta$ ensemble. This has been down in \cite{Webb}, we present the essential steps for completeness. Inspired by (1), we make an educated guess that $D p_j$ is a combination of $p_j$ and $\sum_{0 \leq l < j} p_{l}p_{j-l}$. We observe that
\begin{align}   \nonumber
\sum_{l = 0}^{j}  p_{l}p_{j-l} & = \sum_{l = 0}^{j} (\sum_{a=1}^{n} e^{ilx_a})\cdot (\sum_{b=1}^{n} e^{i(j-l)x_b}) \\ \nonumber
& = \sum_{x_a \neq x_b} \frac{e^{i(j+1)x_a} - e^{i(j+1)x_b} }{e^{ix_a} - e^{ix_b}} + (j+1)p_j\\          \nonumber
& = \sum_{x_a \neq x_b} \frac{e^{ijx_a} - e^{ijx_b} }{e^{ix_a} - e^{ix_b}}(e^{ix_a} + e^{ix_b}) - \sum_{x_a \neq x_b} \frac{e^{i(j-1)x_a} - e^{i(j-1)x_b} }{e^{ix_a} - e^{ix_b}} e^{i(x_a +x_b)}+ (j+1)p_j\ .\\ \nonumber
\end{align}
Notice the symmetry between index a and b of the first term, we have
$$\sum_{x_a \neq x_b} \frac{e^{ijx_a} - e^{ijx_b} }{e^{ix_a} - e^{ix_b}}(e^{ix_a} + e^{ix_b}) = 2 \sum_{x_a \neq x_b} \frac{e^{ijx_a} - e^{ijx_b} }{e^{ix_a} - e^{ix_b}} e^{ix_a}\ .$$
As for the second term, expand the $ e^{i(j-1)x_a} - e^{i(j-1)x_b}$ part, we have
\begin{align}   \nonumber
 \sum_{x_a \neq x_b}\frac{e^{i(j-1)x_a} - e^{i(j-1)x_b} }{e^{ix_a} - e^{ix_b}} e^{i(x_a +x_b)} &= \sum_{x_a \neq x_b} \{e^{i (l-1) x_a} e^{ix_b} + \cdots + e^{ix_a} e^{i(j - 1)x_b}\} \\\ \nonumber
& =  \sum_{x_a \neq x_b}\{\frac{e^{i(j+1)x_a} - e^{i(j+1)x_b} }{e^{ix_a} - e^{ix_b}}- 
e^{ijx_a} - e^{ijx_b}\} \\          \nonumber
& =  \sum_{l = 0}^{j}  p_{l}p_{j-l} -2(n-1)p_j\ . \nonumber
\end{align}
Combining (3) with (1) (2), we find that
\begin{align}   \nonumber
\sum_{l = 0}^{j}  p_{l}p_{j-l} & = \sum_{x_a \neq x_b} \frac{e^{ijx_a} - e^{ijx_b} }{e^{ix_a} - e^{ix_b}} e^{ix_a} + (n+j)p_j \\ \nonumber
& = \sum_{x_a \neq x_b} \frac{e^{ix_a} + e^{ix_b} }{e^{ix_a} - e^{ix_b}} e^{ijx_a} + (n+j)p_j \\          \nonumber
& = - i\sum_{x_a \neq x_b} \cot\left(\frac{x_a - x_b}{2}\right)e^{ijx_a}  + (n+j)p_j  \ .\\ \nonumber
\end{align}
Now we are ready to calculate $D p_j$:
\begin{align}   \nonumber
D p_j & = \frac{\beta}{2} \sum_{x_a \neq x_b}\cot\left(\frac{x_a - x_b}{2}\right) ij e^{ijx_a} - j^2 p_j\\ \nonumber
& = -\frac{j \beta}{2} ( \sum_{l = 0}^{j}  p_{l}p_{j-l} - (n+j)p_j) - j^2 p_j\\          
& = -\frac{j \beta}{2} \sum_{l = 1}^{j-1}  p_{l}p_{j-l} + (\frac{j^2 \beta}{2} - \frac{nj\beta}{2} - j^2)p_j \ .
\end{align}
When $\beta = 2$, we recover formula (1) of lemma 5.5. 

To verify conditions of Theorem 5.3, we need to calculate $D (p_j \overline{p_j})$. From (6.2), it's easy to derive the following formula:
$$D (p_j \overline{p_j}) = -n\beta j p_j \overline{p_j} - (2- \beta)j^2 p_j \overline{p_j} + 2 j^2 n - \frac{\beta}{2} j \sum_{l = 1}^{j-1}p_{-j} p_{l-j} p_j - \frac{\beta}{2} j \sum_{l = 1}^{j-1}p_{l} p_{j-l}\overline{p_j}\ .   \eqno{(6.3)}$$
By (1) and (6.3), we have
\begin{align}        \nonumber
\mathbb{E} [ W_t - W | W] & = t D (\frac{\beta}{2k} p_k \overline{p_k}) + O(t^2)\\            \nonumber
& = t[ \beta k n ( 1 - W) - \frac{\beta (2 - \beta)}{2}\cdot k p_k \overline{p_k} - \frac{\beta^2}{4} \sum_{l = 1}^{k-1}p_l p_{k-l} \cdot \overline{p_k} - \frac{\beta^2}{4} \sum_{l = 1}^{k-1} p_{-l}p_{l-k} \cdot p_k] + O(t^2)  \ .
\end{align}
Therefore, condition 1 of Theorem 5.3  is satisfied with $\Lambda = \beta k n$ and 
$$E = - \frac{\beta (2 - \beta)}{2}\cdot k p_k \overline{p_k} - \frac{\beta^2}{4} \sum_{l = 1}^{k-1}p_l p_{k-l} \cdot \overline{p_k} - \frac{\beta^2}{4} \sum_{l = 1}^{k-1} p_{-l}p_{l-k} \cdot p_k\ .$$
In order to identify $E'$ for the second condition of Theorem 5.3, by (2.3), we have
\begin{align}        \nonumber
\mathbb{E} [ (W_t - W)_i (W_t - W)_j | W ] &= 2 t \langle \nabla W , \nabla  W \rangle + O(t^2)\\            \nonumber
& = t[ 2\beta k n W - \frac{\beta^2}{2} p_{2k}\cdot (p_{-k})^2 - \frac{\beta^2}{2} p_{-2k}\cdot (p_k)^2] + O(t^2)  \ .
\end{align}
 This implies that condition 2 is satisfies with 
$$E' =  - \frac{\beta^2}{2} p_{2k}\cdot (p_{-k})^2 - \frac{\beta^2}{2} p_{-2k}\cdot (p_k)^2\ .$$

To check condition 3 of Theorem 5.3, we must be careful when applying lemma 3.5 to CDBM. Since the drift of CDBM becomes singular when the random particle approaches the boundary of $\sum\nolimits_{n}$, the convergence rate of $\frac{1}{t} \mathbb{E}[ (W_t - W)^2 \mathbb{I}( (W_t - W)^2 > \rho | W ]$ is not uniform. However, from the proof of lemma 3.3, we can extract the following uniform bound:
\begin{align}        \nonumber
\frac{1}{t} \mathbb{E}[ (W_t - W)^2 \mathbb{I}( (W_t - W)^2 > \rho | W ] & \leq \frac{1}{t} \mathbb{E}[ (W_t - W)^2 | W ] \\            \nonumber
& \leq \norm{ \nabla W \cdot \nabla W}_{max} + 2 \norm{W}_{max} \cdot \norm{ D_c W }_{max} \ .\\ \nonumber
\end{align}
From (6.3), $D_c W$ is a smooth function in the closure of $\sum\nolimits_{n}$, so $\norm{ D_c W }_{max}$ is well-defined. The dominated convergence theorem implies that

$$ \lim_{t \rightarrow 0} \frac{1}{t} \mathbb{E} [ | W_t - W |^2  \mathbb{I}( | W_t - W |^2 > \rho ) ] = \mathbb{E} \lim_{t \rightarrow 0}\frac{1}{t} \mathbb{E} [ | W_t - W |^2  \mathbb{I}( | W_t - W |^2 > \rho ) | W]\ .$$

To calculate $\lim_{t \rightarrow 0}\frac{1}{t} \mathbb{E} [ | W_t - W |^2  \mathbb{I}( | W_t - W |^2 > \rho ) | W]$, notice that we can find a $C > 0$ s.t
$$| W_t - W | \leq C \cdot | X_t - X|\ .$$
Define the cut-off radius $ r_X = \min \{ \frac{\sqrt{\rho}}{C}, \frac{1}{2}d \left (X,\partial \Sigma_{n} \right) \}$, where $d \left (X,\partial \Sigma_{n} \right)$ is the distance between $X$ and $\Sigma_{n}$. As before, $\tau_{r_X}$ is the first exit time of ball $B(X ,r_X)$. Then it's obvious that
$$\mathbb{E} [ | W_t - W |^2  \mathbb{I}( | W_t - W |^2 > \rho ) | W] \leq \mathbb{E} [ | W_t - W |^2  \mathbb{I}(\tau_{r_X} \leq t) | W]\ .$$
Furthermore, calculating the exit time from a compact set is a pure local thing. More precisely, we can introduce the stopped process $(U_n)_t = U_{t \wedge \tau_{r_X}}$ and denote the exit time for the stopped process as $\tau_{r_X}'$. Then 
$$\mathbb{P}( \tau_{r_X}' \leq t) = \mathbb{P}( \tau_{r_X}\leq t)\ .$$
Since the drift of the stopped process is bounded, the argument for proving lemma 3.5 applies. We conclude that
$$\lim_{t \rightarrow 0}\frac{1}{t} \mathbb{E} [ | W_t - W |^2  \mathbb{I}( | W_t - W |^2 > \rho ) | W] = 0\ .$$
It follows that
$$\lim_{t \rightarrow 0} \frac{1}{t} \mathbb{E} [ | W_t - W |^2  \mathbb{I}( | W_t - W |^2 > \rho ) ]  = 0\ .$$

The remaining tasks are to estimate the second moment of $E$ and $E'$. Following the notations in \cite{ts}, let $\rho = (\rho_1, \rho_2, \dots )$ be a partition, then the weight of $\rho$ is $|\rho| = \rho_1 + \rho_2 + \cdots$. For a partition $\rho$, $p_{\rho}$ is defined by (2.2) of \cite{ts}. To deal with the $E$ variable, We need to bound terms of the form $\mathbb{E} [ p_{\mu} \overline{p_{\nu}}]$ for partitions $\mu,\nu$ where $|\mu| = |\nu| =2$. Note that for general $\beta$, we don't have precise formula like lemma 5.4. By Theorem 1 of \cite{ts} and some calculation, we get the following upper bound:
\begin{lem} Let $C_E$ be defined as in Theorem 6.1. Then for $0 < l,j \leq k -1 $,
$$|\mathbb{E} [(p_k \overline{p_{k}})^2] \leq \frac{8 C_E}{\beta^2} k^2\ ;$$
$$|\mathbb{E} [p_l \cdot p_{k-l} \cdot p_j \cdot p_{k-j} \cdot \overline{p_k} \cdot \overline{p_k}]| \leq \frac{8 \sqrt{3} C_E}{\beta^3} k^3\ ;$$
$$ |\mathbb{E} [p_{l} \cdot p_{k-l} \cdot p_{k} \cdot \overline{p_k} \cdot \overline{p_k}]| \leq \left( \frac{2}{\beta}\right)^{\frac{5}{2}} C_E \cdot k^2 \sqrt{k}\ ;$$
$$|\mathbb{E} [(p_{-l} \cdot p_{l-k} \cdot \overline{p_k} \cdot p_{j} \cdot p_{k-j} \cdot p_{k}] \leq \frac{4 C_E}{\beta^3} k^3\ .$$
\end{lem}
\noindent From (10),
\begin{align}        \nonumber
\mathbb{E}[E^2] &\leq\ \ \  \frac{\beta^2 (2 - \beta)^2}{4} k^2 (p_k \overline{p_k})^2 + \frac{\beta^4}{16} (\sum_{l=1}^{k-1} p_l p_{k-l} \overline{p_k})^2 +  \frac{\beta^4}{16} (\sum_{l=1}^{k-1} p_{-l} p_{l-k} p_k)^2          \\            \nonumber
& \ \ + \frac{\beta^3 (2 - \beta)}{4} \left(\sum_{l=1}^{k-1}k p_l p_{k-l} p_k (\overline{p_k})^2 + \sum_{l=1}^{k-1} k p_{-l} p_{l-k} \overline{p_k} (p_k)^2 \right) + \frac{\beta^4}{8} (\sum_{l=1}^{k-1} p_{-l} p_{l-k} p_k ) \cdot (\sum_{l=1}^{k-1} p_l p_{k-l} \overline{p_k}) \ .\\ \nonumber
\end{align}
By the previous lemma, we have
\begin{align}        \nonumber
\mathbb{E}[E^2] &\leq 2(2 - \beta)^2 C_E \cdot k^4 + \sqrt{3} \beta C_E k^3 (k-1)^2 + |4-2\beta|\sqrt{2\beta} C_E k^3 (k-1) \sqrt{k} + \frac{\beta C_E}{2} k^3 (k-1)^2  \\            
& \leq 8 \beta^2 C_E \cdot k^5 \ .    \tag{6.4}               
\end{align} 

Now, for the $E'$ variable, we need to  bound $\mathbb{E} [ p_{\mu} \overline{p_{\nu}}]$ for partitions $\mu,\nu$ where $|\mu| = |\nu| =4$. Applying Theorem 1 of \cite{ts} again, we have the following result.
\begin{lem} Let $C_E'$ be defined as in Theorem 6.1. Then for $0 < l,j \leq k -1 $,
$$|\mathbb{E} [p_{2k} \cdot p_{2k} \cdot (p_{-k})^4 ]| \leq \frac{64 \sqrt{3} C_{E'}}{\beta^3} k^3\ ;$$
$$|\mathbb{E} [p_{-2k} \cdot p_{2k} \cdot (p_{-k})^2  \cdot (p_{k})^2]| \leq \frac{32 C_{E'}}{\beta^3} k^3\ .$$
\end{lem}
It follows that
\begin{align}        \nonumber
\mathbb{E}[E'^2] &= \frac{\beta^4}{4} ( p_{-k}^4 \cdot p_{2k} p_{2k} + p_k^4 \cdot p_{ -2k}p_{-2k} + 2 p_{-2k}p_{-k}^2 \cdot p_{2k}  p_k^2) \\            \nonumber
& \leq 32 \sqrt{3} \beta C_{E'} \cdot k^3 + 16 \beta C_{E'} \cdot k^3 \\ \nonumber
& \leq 80 \beta C_{E'} \cdot k^3 \ .      \tag{6.5}
\end{align}
Combining (6.4) (6.5) with Theorem 5.3, we conclude that for $\forall \delta > 0$,
$$d_K (W , Z) \leq \frac{1}{\beta k n \delta} (1+ \frac{2}{e}) \sqrt{8 \beta^2 C_E k^5} + \frac{2}{\beta k n \delta} \sqrt{80 \beta C_{E'}  k^3} + \frac{\delta}{2}\ .$$
Let $\delta$ be the optimal value such that the right hand side achieves the minima, then we get
$$d_K (W , Z) \leq 2 \sqrt{\frac{1}{\sqrt{\beta} n } \sqrt{80 C_{E'} k} + \frac{1}{ \beta n} (1+ \frac{2}{e}) \sqrt{2 C_{E'} k^3}}\ .$$

\noindent{\bf Acknowledgements. }The author is grateful to Prof. Elton Hsu for many helpful  discussions on understanding stochastic calculus on manifolds and reading part of the manuscript.

\end{document}